\documentclass[12pt]{amsart}

\usepackage{xcolor}
\usepackage[vcentermath,enableskew]{youngtab}
\usepackage{amssymb}
\usepackage[all]{xy}

\usepackage{cite}

\newcommand{\Flags}{\mathrm{Fl}}
\newcommand{\Grass}{\mathrm{Gr}}
\newcommand{\OFlags}{\mathrm{FlO}}
\newcommand{\OGrass}{\mathrm{GrO}}
\newcommand{\SFlags}{\mathrm{FlS}}
\newcommand{\SGrass}{\mathrm{GrS}}

\newcommand{\Pic}{\mathrm{Pic}}

\newtheorem{theorem}{Theorem}[section]
\newtheorem{lemma}[theorem]{Lemma}
\newtheorem{proposition}[theorem]{Proposition}

\theoremstyle{definition}
\newtheorem{example}[theorem]{Example}

\newtheorem{definition}[theorem]{Definition}
\newtheorem{remark}[theorem]{Remark}

\newcommand{\GL}{\mathrm{GL}}

\newcommand{\F}{\mathcal{F}}

\usepackage{hyperref}

\begin{document}

\title{On isomorphisms of ind-varieties of generalized flags}
\dedicatory{To the memory of Yuri Manin, an extraordinary human, mathematician, and mentor}
\author{Lucas Fresse}
\address[L. F.]{Universit\'e de Lorraine, CNRS, Institut \'Elie Cartan de Lorraine, UMR 7502, Vandoeu\-vre-l\`es-Nancy, F-54506, France}
\email{lucas.fresse@univ-lorraine.fr}
\author{Ivan Penkov}\thanks{The work of I.\,P. has been partially supported by DFG Grant PE 980/9-1}
\address[I. P.]{Constructor University, 28759 Bremen, Germany}
\email{ipenkov@constructor.university}

\begin{abstract}
Ind-varieties of generalized flags have been studied for two decades. However, a precise statement of when two such ind-varieties, one or both being possibly  ind-varieties of isotropic generalized flags, are isomorphic, has been missing in the literature. Using some recent results on the structure of ind-varieties of generalized flags, we establish a criterion for the existence of an isomorphism as above. Our result claims that, with only two exceptions, isomorphisms of ind-varieties of generalized flags are induced by isomorphisms of respective generalized flags. The exceptional isomorphisms correlate with a well-known result of A.~Onishchik from 1963.
\end{abstract}


\maketitle

\setcounter{tocdepth}{1}


\section{Introduction}

If $X$ and $Y$ are two finite-dimensional flag varieties, possibly one or both being varieties of isotropic flags, the problem of whether $X$ and $Y$ are isomorphic is easily solvable.
One can approach it in different ways, one of which is to look at the automorphism  groups of $X$ and $Y$. This yields an elegant proof of the following theorem, for whose statement we need to introduce some notation.
If $X=\mathrm{Fl}(a_1,\ldots,a_i,V)$ is the variety of flags with dimension sequence $(a_1,\ldots,a_i,a_{i+1}=\dim V)$ in a finite-dimensional vector space $V$ with $\dim V\geq 2$, we say that $X$ is of \emph{general type}. If $X=\mathrm{FlO}(a_1,\ldots,a_i,V)$ is a connected variety
of isotropic flags with dimension sequence $(a_1,\ldots,a_i,a_{i+1}=\dim V)$ in an  orthogonal space $V$ with $\dim V\geq 5$, we say that $X$ is \emph{of orthogonal type}. If $X=\mathrm{FlS}(a_1,\ldots,a_i,V)$ is a (automatically connected) variety of isotropic flags with dimension sequence $(a_1,\ldots,a_i,a_{i+1}=\dim V)$ in a symplectic space $V$ with $\dim V\geq 6$, we say that $X$ is \emph{of symplectic type}.

\begin{theorem}
\label{th-intro}
Let $X$ and $Y$ be two flag varieties of the same type as above. Then $X$ and $Y$ are isomorphic if and only if their dimension sequences coincide, or both $X$ and $Y$ are of general type and their respective dimension sequences $(a_1,\ldots,a_i,a_{i+1})$ and $(b_1,\ldots,b_j,b_{j+1})$ satisfy $i=j$, $a_{i+1}=b_{i+1}$, and $a_k=a_{i+1}-b_k$ for $k\in\{1,\ldots,i\}$.

If $X$ and $Y$ are of different types, then the only possible isomorphisms are as follows:
\begin{itemize}
\item $\mathrm{Fl}(1,\mathbb{C}^{2n})\cong\mathrm{FlS}(1,\mathbb{C}^{2n})$;
\item $\mathrm{FlO}(n-1,\mathbb{C}^{2n-1})\cong\mathrm{FlO}(n,\mathbb{C}^{2n})$.
\end{itemize}
\end{theorem}

This theorem can be considered a corollary of Onishchik's result \cite{O}
claiming  that the connected component of unity of the automorphism group of a flag variety $X$, or  a variety of isotropic flags, is the centerless adjoint group corresponding to the variety, except when $X$ is isomorphic to $\mathrm{FlS}(1,\mathbb{C}^{2n})$
or $\mathrm{FlO}(n-1,\mathbb{C}^{2n-1})$. Indeed, since the algebraic groups $\mathrm{SL}(n)$ for $n\geq 2$, $\mathrm{SO}(m)$ for $m\geq 4$, $\mathrm{Sp}(r)$ for even $r\geq 4$, 
are pairwise non-isomorphic, Onishchik's result reduces the problem to comparing two flag varieties $X$ and $Y$ of the same type for the same vector space $V$. The proof of Theorem \ref{th-intro} gets then easily completed by comparing the grassmannians, or isotropic grassmannians, to which $X$ and $Y$ project.

In the present paper we prove an exact analogue of Theorem \ref{th-intro} for ind-varieties of, possibly isotropic, generalized flags. These ind-varieties are homogeneous ind-spaces for the groups $\GL(\infty)$, $\mathrm{SO}(\infty)$, $\mathrm{Sp}(\infty)$, and have been studied quite extensively in the last twenty years \cite{Dimitrov-Penkov,DPW,IP,HP,PT1,PT}. Nevertheless, a precise statement of when two such ind-varieties are isomorphic has been  missing in the literature. 

First, let us note that Theorem \ref{th-intro} does not imply directly any statement of isomorphism or non-isomorphism of ind-varieties of generalized flags, since two non-isomorphic ind-varieties may admit exhaustions with pairwise isomorphic finite-dimensional varieties, and conversely, an ind-variety may admit two exhaustions by pairwise non-isomorphic finite-dimensional varieties.
Next, we recall that the automorphism groups of ind-varieties of, possibly isotropic, generalized flags have been computed in \cite{IP}. However, since the question of when two such groups are isomorphic as abstract groups has not yet been addressed (and may be quite hard), we are unable to produce an argument as direct as in the outline of proof of Theorem \ref{th-intro} given above. Instead, we rely on some basic information about automorphism groups of ind-varieties of generalized flags and on a
technique developed in the papers \cite{PT1,PT}. This technique turns out to be very useful also in the problem of isomorphisms.

The precise statement of our main result is Theorem \ref{T-main} below.
In Section \ref{section-3} we have collected preliminaries on finite-dimensional flag varieties. Sections \ref{section-4} and \ref{section-5} are devoted to the proof of Theorem \ref{T-main}.


\subsection*{Acknowledgement} We thank Valdemar Tsanov for providing with the reference \cite{Landsberg-Manivel} and
explaining its relevance.

\section{Statement of result}

\label{section-2}

\subsection{Short review of ind-varieties of generalized flags}

The base field is the field of complex numbers $\mathbb{C}$. Let $V$ be a countable-dimensional vector space, possibly equipped with an orthogonal (i.e., non-degenerate, symmetric) or symplectic (i.e., non-degenerate, antisymmetric) bilinear form $\omega$. By $E$ we denote a basis of $V$. In the presence of a form $\omega$, we make the following definition.

\begin{definition}
\label{D2.1}
Assume that $V$ is equipped with an orthogonal or symplectic form $\omega$.
A basis $E$ is said to be {\em isotropic} if it is equipped with an involution
$i_E:E\to E$ with at most one fixed point, such that $\omega(e,e')\not=0$ if and only if $e'=i_E(e)$.
Then:
\begin{itemize}
\item If $\omega$ is orthogonal and $i_E$ has one fixed point, then the basis $E$ is said to be of type B.
\item If $\omega$ is orthogonal and $i_E$ has no fixed point, then $E$ is said to be of type D.
\item If $\omega$ is symplectic, then $i_E$ cannot have a fixed point, and the basis $E$ is said to be of type C.
\end{itemize}
\end{definition}

In \cite{Dimitrov-Penkov}, the homogeneous spaces of the form $\mathbf{G}/\mathbf{P}$ have been described, where $\mathbf{G}$ is one of the classical ind-groups $\mathrm{SL}(\infty)$ (or $\GL(\infty)$), $\mathrm{SO}(\infty)$, $\mathrm{Sp}(\infty)$,
and $\mathbf{P}\subset \mathbf{G}$ is a splitting parabolic subgroup.
The term ``splitting'' means that $\mathbf{P}$ contains the Cartan subgroup of transformations inside $\mathbf{G}$ which are diagonal in some basis $E$ (isotropic in the case of $\mathrm{SO}(\infty)$ and $\mathrm{Sp}(\infty)$) of the underlying space $V$.


The description is by means of the notion of generalized flag.

\begin{definition}
(a) A {\em generalized flag} of $V$ is a collection $\mathcal{F}$ of subspaces of $V$ which is totally ordered by inclusion and such that
\begin{itemize}
    \item every $F\in\mathcal{F}$ has an immediate predecessor $F'$ or an immediate successor $F''$ in $\mathcal{F}$;
    \item every vector $v\in V\setminus\{0\}$ belongs to $F''\setminus F'$ for a unique pair of consecutive subspaces $(F',F'')$ of $\mathcal{F}$.
\end{itemize}

(b) In the case where $V$ is equipped with an orthogonal of symplectic form $\omega$,
we say that $\mathcal{F}$ is {\it isotropic} if there is an involution $i_\mathcal{F}:\mathcal{F}\to\mathcal{F}$ such that $i_\mathcal{F}(F)=F^\perp$ for all $F\in\mathcal{F}$, where $F^\perp$ stands for the orthogonal subspace to $F$ with respect to $\omega$.

(c) If $E$ is a basis of $V$, then $\mathcal{F}$ is said to be {\em compatible with $E$} if 
each subspace $F\in\mathcal{F}$ has a basis formed by elements of $E$.
We say that $\mathcal{F}$ is {\em weakly compatible with $E$} if it is compatible with some basis $E'$ which differs from $E$ by finitely many vectors, i.e.,
$\# E\setminus(E\cap E')=\# E'\setminus(E\cap E')<+\infty$.
\end{definition}

In \cite{Dimitrov-Penkov}, an equivalence relation called \emph{$E$-commensurability} is introduced on generalized flags.
Then, given a generalized flag $\mathcal{F}$ compatible with a basis $E$, one defines $\Flags(\mathcal{F},E,V)$ as the set of all generalized flags in $V$ which are $E$-commensurable with $\mathcal{F}$. 
If $\mathcal{F}$ and $E$ are isotropic, one defines instead $\Flags_\omega(\mathcal{F},E,V)$ as the set of all isotropic generalized flags in $V$ which are $E$-commensurable with $\mathcal{F}$.
It is known that 
$\Flags(\mathcal{F},E,V)$ and $\Flags_\omega(\mathcal{F},E,V)$ have natural structures of ind-varieties. We will recall these structures later on. In what follows, whenever we write $\mathrm{Fl}(\mathcal{F},E,V)$ or  $\Flags_\omega(\mathcal{F},E,V)$, we assume that the generalized flag $\mathcal{F}$ is compatible with the basis $E$.

We will adopt the following notation:
\begin{itemize}
\item If $\omega$ is an orthogonal form on $V$, then 
$E$ is of type B or D, and we set in both cases
$\OFlags(\mathcal{F},E,V):=\mathrm{Fl}_\omega(\mathcal{F},E,V)$, with the following exception. If $\mathcal{F}$ contains a subspace $F$ such that $F^\perp=F$ or a subspace $F'$ such that $\dim F'^\perp/F'=2$ (in both cases $E$ has to be of type D), then $\Flags_\omega(\mathcal{F},E,V)$ consists of two isomorphic connected components, and we define $\OFlags(\mathcal{F},E,V)$ as either one. In addition, if there is $F'\in\mathcal{F}$ with $\dim F'^\perp/F'=2$, we assume that $\mathcal{F}$ contains also one of the two maximal isotropic subspaces $F$ containing $F'$.
\item If $\omega$ is a symplectic form on $V$, then $E$ is of type C and we set
$\SFlags(\mathcal{F},E,V)=\mathrm{Fl}_\omega(\mathcal{F},E,V)$.

\end{itemize}

Ind-grassmannians correspond to quotients $\mathbf{G}/\mathbf{P}$ with $\mathbf{P}$ maximal:
\begin{itemize}
\item If $\F=\{\{0\}\subset F\subset V\}$ then we set $\Grass(F,E,V)=\Flags(\F,E,V)$.
\item In the case where $V$ is equipped with an orthogonal
or symplectic form $\omega$, a minimal isotropic generalized flag has the form $\F=\{\{0\}\subset F\subset F^\perp\subset V\}$ (where $F$ and $F^\perp$ may coincide), and we set 
$\OGrass(F,E,V)=\OFlags(\mathcal{F},E,V)$ or $\SGrass(F,E,V)=\SFlags(\mathcal{F},E,V)$
depending on whether $\omega$ is orthogonal or symplectic.
\item In accordance with our convention above, when using the notation $\OGrass(F,E,V)$, we exclude the case when $\dim F^\perp/F=2$. Instead, we consider 
$\OFlags(\mathcal{F},E,V)$ where $\mathcal{F}=\{\{0\}\subset F\subset \tilde{F}\subset F^\perp\subset V\}$, $\tilde{F}$ being one of the two maximal isotropic spaces containing $F$.
\end{itemize}

\begin{remark} \label{remark-BDgrass}
If $V$ is an orthogonal space, then $V$ admits isotropic bases $E_1$ and $E_2$ of respective types B and D. Accordingly, maximal isotropic subspaces  $F$ of $V$ are of two types:
either $\dim F^\perp/F=1$ or $F^\perp=F$.
As we will see below, the corresponding ind-grassmannians
$\OGrass(F_1,E_1,V)$ for $\dim F_1^\perp/F_1=1$ and $\OGrass(F_2,E_2,V)$ for $F_2^\perp=F_2$ are isomorphic as ind-varieties.
This property is an infinite-dimensional analogue of the isomorphism stated in the second bullet point of Theorem \ref{th-intro}.
\end{remark}

\subsection{Main result} \label{section-2.2}

\begin{definition}
\label{definition-isomorphic}
(a) 
Let $\mathcal{F}$ and $\mathcal{G}$ be generalized flags of countable dimensional spaces $V$ and $W$, respectively. Without further assumption, we say that $\mathcal{F}$ and $\mathcal{G}$ are \emph{isomorphic} if there exists a linear isomorphism $\phi:V\to W$ such that $\mathcal{G}=\{\phi(F):F\in\mathcal{F}\}$.

(b) In the case where $V$ and $W$ are equipped with symplectic forms (resp., orthogonal forms) $\omega$ and $\omega'$, we assume that $\mathcal{F}$ and $\mathcal{G}$ are isotropic generalized flags and say that they are \emph{isomorphic}  if the isomorphism $\phi$
preserves the forms: $\omega'(\phi(x),\phi(y))=\omega(x,y)$ for all $(x,y)\in V\times V$.
\end{definition}

If $\mathcal{F}=\{F_\theta:\theta\in \Theta\}$ is a generalized flag in a countable-dimensional vector space $V$, compatible with a basis $E$ of $V$, then we define its orthogonal  as the chain $\mathcal{F}^\perp=\{F_\theta^\perp:\theta\in\Theta\}$
where $F_\theta^\perp$ is the annihilator of $F_\theta$ in the space $\langle E^*\rangle$, and $E^*$ denotes the system of linear functionals on $V$ dual to the basis $E$. If $V$ is equipped with an orthogonal or a symplectic form $\omega$ and the basis $E$ is isotropic, then we use this form to identify $V$ and $\langle E^*\rangle$. Moreover, the above definition of an isotropic generalized flag $\mathcal{F}$ is equivalent to the requirement $\mathcal{F}=\mathcal{F}^\perp$.

\begin{theorem} \label{T-main}
Let $X$ and $Y$ be ind-varieties of, possibly isotropic, generalized flags as above. In other words, $X=\Flags(\mathcal{F},E,V)$, or $X=\OFlags(\mathcal{F},E,V)$, or $X=\SFlags(\mathcal{F},E,V)$, and similarly
$Y=\Flags(\mathcal{G},E',W)$, or $Y=\OFlags(\mathcal{G},E',W)$, or $Y=\SFlags(\mathcal{G},E',W)$.
Then the ind-varieties $X$ and $Y$ are isomorphic whenever $\mathcal{F}$ and $\mathcal{G}$, or $\mathcal{F}$ and $\mathcal{G}^\perp$, are isomorphic, possibly as isotropic flags.

The only additional isomorphisms $X\cong Y$ are the following:
\begin{itemize}
\item $X=\Grass(F,E,V)$, $Y=\SGrass(G,E',W)$, where $\dim F=\dim G=1$;
\item $X=\OGrass(F,E,V)$, $Y=\OGrass(G,E',W)$, where $\dim F^\perp/F=1$ and $G^\perp=G$.
\end{itemize}
\end{theorem}


\begin{remark}
Note that two ind-varieties  $\OFlags(\mathcal{F}_1,E_1,V)$
and $\OFlags(\mathcal{F}_2,E_2,V)$, where $E_1$ is an isotropic basis of type B and $E_2$ is an isotropic basis of type D, may be isomorphic also beyond the special case of Remark \ref{remark-BDgrass}. This is a consequence of the observation that a given isotropic generalized flag $\mathcal{F}$ may be compatible with two different isotropic bases $E_1$ and $E_2$ of respective types B and D, as illustrated by the following example. 
\end{remark}

\begin{example} \label{E2.7}
Consider an isotropic flag $\mathcal{F}$ in an orthogonal space $V$, with the property that $W_\mathcal{F}:=\sum\limits_{\atop^{F\in\mathcal{F}}_{F\subset F^\perp}}F$ has infinite codimension in its orthogonal. Then there exist two isotropic bases $E_1$ and $E_2$ of respective types B and D with which $\mathcal{F}$ is compatible
and $E_1\cap W_\mathcal{F}=E_2\cap W_\mathcal{F}$. Consequently, $\OFlags(\mathcal{F},E_1,V)=\OFlags(\mathcal{F},E_2,V)$.

\end{example}







\section{A review on embeddings of flag varieties}

\label{section-3}

Throughout this section, $V,V',W,W'$ are finite-dimensional vector spaces.

\subsection{Short review of Picard groups for flag varieties}

For an integer $0< p< \dim V$, we denote by $\Grass(p;V)$ the grassmannian of $p$-dimensional subspaces in $V$. 
It can be realized as a projective variety via the Pl\"ucker embedding $\pi:\Grass(p;V)\hookrightarrow \mathbb{P}(\bigwedge^p V)$. Moreover, the Picard group $\Pic(\Grass(p;V))$ of $\Grass(p;V)$ is isomorphic to $(\mathbb{Z},+)$, and its generators are $\mathcal{O}_{\Grass(p;V)}(1):=\pi^*\mathcal{O}_{\mathbb{P}(\bigwedge^pV)}(1)$ and $\mathcal{O}_{\Grass(p;V)}(-1):=\pi^*\mathcal{O}_{\mathbb{P}(\bigwedge^pV)}(-1)$. 
Here $\mathcal{O}_{\mathbb{P}(\bigwedge^pV)}(-1)$ stands for the tautological bundle of $\mathbb{P}(\bigwedge^pV)$ and $\mathcal{O}_{\mathbb{P}(\bigwedge^pV)}(1)$ stands for its dual.

For a sequence of integers $0<p_1<\ldots<p_k<\dim V$, we denote by $\Flags(p_1,\ldots,p_k;V)$ the variety of (partial) flags
\[
\Flags(p_1,\ldots,p_k;V)=\{(V_1,\ldots,V_k)\in\Grass(p_1;V)\times\cdots\times\Grass(p_k;V):V_1\subset\ldots\subset V_k\}.
\]
We have 
\[\Pic(\Flags(p_1,\ldots,p_k;V))\cong \mathbb{Z}^{k}.\]
More precisely, if we denote by $L_i$ the pull-back 
\[L_i=\mathrm{pr}_i^*\mathcal{O}_{\Grass(p_i;V)}(1)\]
along the projection 
\[\mathrm{pr}_i:\Flags(p_1,\ldots,p_k;V)\to \Grass(p_i;V)\]
(for $i=1,\ldots,k$),
then $[L_1],\ldots,[L_{k}]$ is a set of generators of the Picard group, which we will refer to as the set of {\em preferred generators}.

If $V$ is a vector space endowed with an orthogonal or symplectic form $\omega$, we assume that the sequence $(p_1,\ldots,p_k)$ satisfies 
$$p_i+p_{k-i+1}=\dim V\quad\mbox{ for all $i=1,\ldots,k$}.$$ 

\bigskip

{\bf Orthogonal case:} Here we assume that the form $\omega$ is orthogonal.
If $\frac{\dim V}{2},\frac{\dim V}{2}-1\notin\{p_1,\ldots,p_k\}$ (which is automatic when $\dim V$ is odd), we define $\OFlags(p_1,\ldots,p_k;V)\subset\Flags(p_1,\ldots,p_k;V)$ as the subvariety of {\it isotropic flags}, i.e., flags $(F_1\subset\ldots\subset F_k)$ such that $F_i^\perp= F_{k-i+1}$ for all $i$. If $\frac{\dim V}{2}\in\{p_1,\ldots,p_k\}$ or $\frac{\dim V}{2}-1\in\{p_1,\ldots,p_k\}$, the subvariety of $\Flags(p_1,\ldots,p_k;V)$ of isotropic flags consists of two irreducible components, and we define $\OFlags(p_1,\ldots,p_k;V)$ as either of these two components.

Moreover, as it is well known every isotropic subspace of dimension $\frac{\dim V}{2}-1$ is contained in exactly two Lagrangian subspaces, so we lose no generality in considering only sequences $(p_1,\ldots,p_k)$ which satisfy the  condition
$$
\frac{\dim V}{2}-1\in\{p_1,\ldots,p_k\}\quad\Rightarrow\quad \frac{\dim V}{2}\in\{p_1,\ldots,p_k\}.
$$

We denote $\OGrass(p;V):=\OFlags(p,\dim V-p;V)$ if $p\notin\{\frac{\dim V}{2},\frac{\dim V}{2}-1\}$
and $\OGrass(\frac{\dim V}{2};V):=\OFlags(\frac{\dim V}{2};V)$, assuming that $\dim V$ is even in the latter case.
We do not define an orthogonal grassmannian for $p=\frac{\dim V}{2}-1$ as we consider instead $\OFlags(\frac{\dim V}{2}-1,\frac{\dim V}{2},\frac{\dim V}{2}+1;V)$.

Let $\ell=\lfloor\frac{k}{2}\rfloor$. 
If $\frac{\dim V}{2}-1\notin\{p_1,\ldots,p_k\}$, then the pull-backs $L_i:=\mathrm{pro}_i^*\mathcal{O}_{\OGrass(p_i;V)}(1)$
by the projections $\mathrm{pro}_i:\OFlags(p_1,\ldots,p_k;V)\to\OGrass(p_i;V)$, for $i\in\{1,\ldots,\ell\}$, is a set of generators of the Picard group $\Pic\,\OFlags(p_1,\ldots,p_k;V)$, which we call \emph{preferred generators}.
If $\frac{\dim V}{2}-1\in\{p_1,\ldots,p_k\}$, that is $\frac{\dim V}{2}-1=p_{\ell-1}$, then the preferred generators $L_i$ of $\Pic\,\OFlags(p_1,\ldots,p_k;V)$ are  as above except for $i=\ell-1$, and the $(\ell-1)$-th preferred generator is by definition
$(\bigwedge^{\frac{\dim V}{2}-1}S_{\ell-1})^*$ where $S_{\ell-1}$ is the tautological bundle of rank $\frac{\dim V}{2}-1$ on $\OFlags(p_1,\ldots,p_k;V)$.

\bigskip

{\bf Symplectic case:} 
If the form $\omega$ is symplectic, we denote by $\SFlags(p_1,\ldots,p_k;V)\subset\Flags(p_1,\ldots,p_k;V)$ the subvariety of isotropic flags. Moreover, we set $\SGrass(p;V):=\SFlags(p,\dim V-p;V)$ if $\dim V\not=2p$, and $\SGrass(\frac{\dim V}{2};V):=\SFlags(\frac{\dim V}{2};V)$.

Let $\ell=\lfloor\frac{k}{2}\rfloor$. Then
$\Pic\,\SFlags(p_1,\ldots,p_k;V)\cong\mathbb{Z}^\ell$, and the pull-backs $L_i:=\mathrm{prs}_i^*\mathcal{O}_{\SGrass(p_i;V)}(1)$ by the projections $\mathrm{prs}_i:\SFlags(p_1,\ldots,p_k;V)\to\SGrass(p_i;V)$ for $i\in\{1,\ldots,\ell\}$
yield a set of generators $[L_1],\ldots,[L_{\ell}]$ of $\Pic\,\SFlags(p_1,\ldots,p_k;V)$, which again we refer to as \emph{preferred generators}.

\bigskip

We close this subsection with the following well-known fact.

\begin{lemma}
\label{L3.1}
Let $\mathcal{M}$ be a line bundle on $\Flags(p_1,\ldots,p_k;V)$, $\OFlags(p_1,\ldots,p_k;V)$, or $\SFlags(p_1,\ldots,p_k;V)$, and assume that the equality 
$$[\mathcal{M}]=n_1[L_1]+\ldots+n_{k}[L_{k}]\quad\mbox{with $n_1,\ldots,n_{k}\in\mathbb{Z}$}$$
holds in the Picard group, where $[L_1],\ldots,[L_\ell]$ is the respective set of preferred generators. Then the following conditions are equivalent:
\begin{itemize}
\item[\rm (i)] $\mathcal{M}$ is very ample; 
\item[\rm (ii)] $\mathcal{M}$ is ample;
\item[\rm (iii)] $n_i>0$ for all $i\in\{1,\ldots,k\}$.
\end{itemize}
\end{lemma}

\subsection{Embeddings of flag varieties}

In this section, we denote by $X$ one of the flag varieties
$$\Flags(p_1,\ldots,p_K;V),\quad\SFlags(p_1,\ldots,p_{K};V),\quad \OFlags(p_1,\ldots,p_{K};V),$$ 
and by $Y$ a respective flag variety
$$\Flags(q_1,\ldots,q_L;W),\quad\SFlags(q_1,\ldots,q_{L};W),\quad\OFlags(q_1,\ldots,q_{L};W)$$
of the same type as $X$.  
Consider an embedding (i.e. closed immersion) of flag varieties
\[\varphi:X\hookrightarrow Y,\]
together with the group homomorphism on Picard groups
\[\varphi^*:\Pic\,Y\to\Pic\,X\]
which it induces. 
Let
$[L_1],\ldots,[L_{k}]$ and $[M_1],\ldots,[M_{\ell}]$ be the respective sets of preferred generators of $\Pic\,X$ and $\Pic\,Y$ (in the sense of the previous subsection), where $k=K$ and $\ell=L$, or $k=\lfloor\frac{K+1}{2}\rfloor$ and $\ell=\lfloor\frac{L+1}{2}\rfloor$, depending on whether a flag variety of general type or a variety of isotropic flags is considered.

\begin{lemma}
\label{L-3.2}
For all $j\in\{1,\ldots,\ell\}$, we have $\varphi^*([M_j])\in \mathbb{Z}_{\geq 0}[L_1]+\ldots+\mathbb{Z}_{\geq 0}[L_{k}]$.
\end{lemma}

\begin{proof}
Since $\varphi$ is an embedding, if $\mathcal{M}$ is an ample line bundle on $Y$ then $\varphi^*\mathcal{M}$ should be an ample line bundle on $X$. In view of Lemma \ref{L3.1}, we must have
\[\varphi^*(\mathbb{Z}_{>0}[M_1]+\ldots+\mathbb{Z}_{>0}[M_{\ell}])\subset \mathbb{Z}_{>0}[L_1]+\ldots+\mathbb{Z}_{>0}[L_{k}].\]
The claim of the lemma follows.
\end{proof}

We now recall from \cite{PT} the notion of linear embedding, standard extension, and factorization through direct product.

\begin{definition}
\label{D-linear}
Let $\varphi:X\hookrightarrow Y$ be an embedding of flag varieties
as above.

(a) We say that $\varphi$
is {\em linear} if 
\[\varphi^*[M_j]=0\quad\mbox{or}\quad\varphi^*[M_j]\in\{[L_1],\ldots,[L_{k}]\}\]
for all $j\in\{1,\ldots,\ell\}$. 

(b.1) We say that $\varphi$ is a {\em strict standard extension} if there are 
\begin{itemize}
\item a linear monomorphism $\alpha:V\hookrightarrow W$ and a decomposition $W=\mathrm{Im}\,\alpha\oplus K$;
\item a nondecreasing sequence of subspaces $K_0=\{0\}\subset K_1\subset K_2\subset\ldots\subset K_\ell=K$;
\item a surjective, nondecreasing map $\kappa:\{0,1,\ldots,\ell\}\to\{0,1,\ldots,k\}$ such that, for all $i\in\{1,\ldots,\ell\}$, $K_{i-1}= K_i\Rightarrow\kappa(i-1)<\kappa(i)$;
\item in the case where $V$ and $W$ are equipped with nondegenerate symmetric or antisymmetric forms $\omega$ and $\phi$, respectively, then the monomorphism $\alpha$ is compatible with the forms in the sense that $\phi(\alpha(v_1),\alpha(v_2))=\omega(v_1,v_2)$ and the decomposition $W=\mathrm{Im}\,\alpha\oplus K$ is orthogonal;
\end{itemize}
so that $\varphi$ can be expressed as
\[\varphi:(F_0=\{0\},F_1,\ldots,F_k)\mapsto (\alpha(F_{\kappa(1)})\oplus K_1,\ldots,\alpha(F_{\kappa(\ell)})\oplus K_\ell).\] 

(b.2) When $X=\Flags(p_1,\ldots,p_k;V)$ and $Y=\Flags(q_1,\ldots,q_\ell;W)$, we say that $\varphi$ is a {\em modified standard extension} if $\varphi$ equals the composition $\delta\circ\tilde\varphi$ of a strict standard extension $\tilde\varphi:X\hookrightarrow Y^\vee:=\mathrm{Fl}(\dim W-q_\ell,\ldots,\dim W-q_1;W^*)$ with the isomorphism
\[\delta:Y^\vee\to Y,\ (Z_1,\ldots,Z_\ell)\mapsto (Z_{\ell}^\perp,\ldots,Z_1^\perp).\]

(b.3) We say that $\varphi$ is a {\em standard extension} if $\varphi$ is a strict or a modified standard extension.

(c) We say that $\varphi$ \textit{factors through a direct product} if there are $s\geq 2$, a decomposition 
$\{p_1,\ldots,p_{k}\}=R_1\sqcup\ldots\sqcup R_s$ into nonempty subsets, and exponents $t_1,\ldots,t_s\geq 1$ such that $\varphi$ factors as the composition
\[
X\stackrel{\psi_R}{\hookrightarrow } \prod_{i=1}^s\Flags'(R_i;V)^{t_i}\stackrel{\psi}{\hookrightarrow} Y
\]
where $\psi_R$ is the canonical embedding
and $\psi$ is an embedding, and the notation $\Flags'$ means $\Flags$ or $\Flags_\omega$ depending on whether $X$ is consists of general or isotropic flags.

(d.1) Assume that $W$ is endowed with an orthogonal or, respectively, symplectic form so that $V$ is an isotropic subspace of $W$. Then, there are natural embeddings
$$
X=\Flags(p_1,\ldots,p_k;V)\hookrightarrow \OFlags(p_1,\ldots,p_k;W)\quad\mbox{and}\quad
X^\vee\hookrightarrow \OFlags(p_1,\ldots,p_k;W),
$$
respectively,
$$
X=\Flags(p_1,\ldots,p_k;V)\hookrightarrow \SFlags(p_1,\ldots,p_k;W)\quad\mbox{and}\quad
X^\vee\hookrightarrow \SFlags(p_1,\ldots,p_k;W),
$$
which we call {\it isotropic extensions}.

(d.2) A {\it combination of standard and isotropic extensions} is an embedding of the form
\begin{eqnarray*}
\OFlags(p_1,\ldots,p_k;V)\stackrel{t}{\hookrightarrow}  \Flags(p_1,\ldots,p_k;V)
\stackrel{\zeta}{\hookrightarrow} \Flags(q_1,\ldots,q_\ell;V') \qquad\qquad\qquad
 \\ \stackrel{\chi}{\hookrightarrow} \OFlags(q_1,\ldots,q_\ell;W)
\stackrel{\xi}{\hookrightarrow} \OFlags(r_1,\ldots,r_m;W'),
\end{eqnarray*}
respectively,
\begin{eqnarray*}
\SFlags(p_1,\ldots,p_k;V)\stackrel{t}{\hookrightarrow} \Flags(p_1,\ldots,p_k;V)
\stackrel{\zeta}{\hookrightarrow} \Flags(q_1,\ldots,q_\ell;V') \qquad\qquad\qquad
 \\ \stackrel{\chi}{\hookrightarrow} \SFlags(q_1,\ldots,q_\ell;W)
\stackrel{\xi}{\hookrightarrow} \SFlags(r_1,\ldots,r_m;W'),
\end{eqnarray*}
where $t$ is the tautological embedding, $\zeta,\xi$ are standard extensions, and $\chi$ is an isotropic extension.
\end{definition}

\begin{remark}
\label{remark-3.4}
Given a standard extension $\varphi:X\to Y$, the splitting $W=\mathrm{Im}\,\alpha\oplus K$ is in general not unique if $X$ and $Y$ are of general type, and is unique if $X$ and $Y$ are both of orthogonal or symplectic type, since in the latter cases this splitting is assumed orthogonal.
\end{remark}

The following proposition is based on \cite[Proposition 2.3]{PT}.

\begin{proposition}
\label{prop-linear}
Let $\varphi:X=\Flags(p_1,\ldots,p_k;V)\hookrightarrow Y=\Flags(q_1,\ldots,q_\ell;W)$ be an embedding of  flag varieties.

The following conditions are equivalent.
\begin{itemize}
\item[\rm (i)] The embedding $\varphi$ is linear.
\item[\rm (ii)] There are 
\begin{itemize}
\item[$\bullet$] a partition $\{1,\ldots,\ell\}=I_0\sqcup I_1\sqcup\ldots\sqcup I_{k}$, with $I_i\not=\emptyset$ for $i\not=0$, 
\item[$\bullet$] a sequence of linear embeddings $\varphi[i]=(\varphi_{i,j})_{j\in I_i}:\Grass(p_i;V)\hookrightarrow \prod_{j\in I_i}\Grass(q_j;W)$, for $0\leq i\leq k$, and if $I_0\not=\emptyset$ a constant map $X_0:=\{\mathrm{pt}\}\hookrightarrow\prod_{j\in I_0}\Grass(q_j;W)$ \end{itemize}
such that the following diagram commutes
\[
\xymatrix{X=\Flags(p_1,\ldots,p_k;V) \ar@{^{(}->}[r]^\varphi \ar@{^{(}->}[d]^\mu& Y=\Flags(q_1,\ldots,q_\ell;W) \ar@{^{(}->}[d]^\pi \\
X_0\times \prod_{i=1}^{k}\Grass(p_i;V) \ar@{^{(}->}[r]^{\quad\prod\varphi[i]} & \prod_{j=1}^{\ell}\Grass(q_j;W), }
\]
where the vertical arrows are the natural embeddings. 
\end{itemize}
A similar result holds in the symplectic and orthogonal cases.
\end{proposition}

\begin{proof}
(i)$\Rightarrow$(ii) is shown in \cite[Proposition 2.3]{PT}. 
(ii)$\Rightarrow$(i): for every $j\in\{1,\ldots,\ell\}$, assuming that $j\in I_i$ with $i\not=0$, we have
\begin{eqnarray*}
[(\pi\circ\varphi)^*\mathrm{pr}_j^*\mathcal{O}_{\Grass(q_j;W)}(1)] & = & 
[\mu^*\mathrm{pr}_i^*\varphi[i]^*\mathrm{pr}_j^*\mathcal{O}_{\Grass(q_j;W)}(1)] \\
 & \in & \{0,[\mu^*\mathrm{pr}_i^*\mathcal{O}_{\Grass(p_i;V)}(1)]\}=\{0,[L_i]\}
\end{eqnarray*}
by the assumption that $\varphi[i]$ is linear. If $j\in I_0$, then
\[[(\pi\circ\varphi)^*\mathrm{pr}_j^*\mathcal{O}_{\Grass(q_j;W)}(1)]= 
[\mu^*\mathrm{pr}_0^*\varphi[0]^*\mathrm{pr}_j^*\mathcal{O}_{\Grass(q_j;W)}(1)]=0.\] The conclusion follows.
\end{proof}

A key result is now the following:

\begin{theorem}[{\cite[Theorem 1]{PT1}}, {\cite[Theorem 4.2]{PT}}]
\label{T-PT}
{\rm (a)} Let $\varphi:X\hookrightarrow Y$ be an embedding of flag varieties of the same type. Assume that $\varphi$ is linear, does not factor through a direct product, and all the maps $\varphi[i]$ of Proposition \ref{prop-linear} are standard extensions. Then $\varphi$ is a standard extension. 

{\rm (b)} Assume that $X$ and $Y$ are grassmannians of the same type. In addition,
in the orthogonal case suppose that $X$ and $Y$ are of the form $\OGrass(p;V)$ with $p\notin\{\frac{\dim V}{2}-1,\frac{\dim V}{2}\}$.  
\begin{itemize}
\item[\rm (i)] If $X$ and $Y$ are of general type, then $\varphi:X\hookrightarrow Y$ is a standard extension if and only if is linear and does not factor through a projective space.
\item[\rm (ii)] If $X$ and $Y$ are of orthogonal or symplectic type, then $\varphi:X\hookrightarrow Y$ is a standard extension or a combination of standard and isotropic extensions if and only if $\varphi$ is linear and does not factor through a projective space and, in the orthogonal case, also not through a quadric.
\end{itemize}
\end{theorem}


\subsection{Additional lemmas}

\begin{lemma}
\label{L-se}
{\rm (a)} The composition of two standard extensions is a standard extension.

{\rm (b)} The composition of two standard extensions $\varphi_1$ and $\varphi_2$ is strict if and only if $\varphi_1$ and $\varphi_2$ are both strict or are both modified.
\end{lemma}

\begin{proof}
Straightforward.
\end{proof}

\begin{lemma}
\label{L-diagram}
Let 
\[\xymatrix{X=\mathrm{Fl}(n_1,\ldots,n_k;V) \ar@{^{(}->}[rr]^\chi \ar@{^{(}->}[rd]^\varphi & & \mathrm{Fl}(q_1,\ldots,q_m;U)=Z \\
 & Y=\mathrm{Fl}(p_1,\ldots,p_\ell;W) \ar@{^{(}->}[ru]^\psi}\]
be a commutative diagram of strict standard extensions.
Assume that 
\begin{itemize}
\item $\varphi$ corresponds to $\alpha:V\hookrightarrow W$, a decomposition $W=\mathrm{Im}\,\alpha\oplus K$, a nondecreasing sequence of subspaces $K_0=\{0\}\subset K_1\subset\ldots\subset K_\ell=K$, and a surjective map $\kappa:\{0,\ldots,\ell\}\to\{0,\ldots,k\}$, in the sense of Definition \ref{D-linear}\,(b.1);
\item $\psi$ corresponds similarly to $\beta:W\hookrightarrow U$, $U=\mathrm{Im}\,\beta\oplus L$, $L_0\subset\ldots\subset L_m=L$, $\lambda:\{0,\ldots,m\}\to\{0,\ldots,\ell\}$;
\item $\chi$ corresponds similarly to $\gamma:V\hookrightarrow U$, $U=\mathrm{Im}\,\gamma\oplus M$, $M_0\subset\ldots\subset M_m=M$, $\mu:\{0,\ldots,m\}\to\{0,\ldots,k\}$.
\end{itemize}
Then we have $\mu=\kappa\circ\lambda$, $M_i=L_i\oplus\beta(K_i)$ for all $i\in\{1,\ldots,m\}$, and up to modifying $\beta$ we can assume that $\gamma=\beta\circ\alpha$.

Similar statements hold in the symplectic and orthogonal cases, and the equality $\gamma=\beta\circ\alpha$ always holds in this case.
\end{lemma}

\begin{proof}
Since $\chi=\psi\circ\varphi$, for all $\mathcal{F}=(F_1,\ldots,F_k)\in X$, all $i\in\{1,\ldots,m\}$, we have
\begin{equation}
\label{1-L-diagram}
\gamma(F_{\mu(i)})\oplus M_i=\beta\alpha(F_{\kappa\lambda(i)})\oplus \beta(K_{\lambda(i)})\oplus L_i.
\end{equation}
It follows from the equality $\bigcap_{\mathcal{F}\in X}F_{\mu(i)}=\{0\}$ that
\[M_i=\beta(K_{\lambda(i)})\oplus L_i\quad\mbox{for all $i=1,\ldots,m$}.\]
Then, since the dimensions of $F_1,\ldots,F_k$ are pairwise distinct, formula (\ref{1-L-diagram}) implies for dimension reasons
\[\mu(i)=\kappa\circ\lambda(i)\quad\mbox{for all $i=1,\ldots,m$}.\]
Take $i_0\in\{1,\ldots,m\}$ minimal such that $\mu(i_0)\not=0$. Then $\bigcup_{\mathcal{F}\in X}F_{\mu(i_0)}=V$ and we must have
\[\mathrm{Im}\,\gamma\oplus M_{i_0}=\mathrm{Im}\,\beta\circ\alpha\oplus M_{i_0}.\]
Up to replacing $\beta$ by some $\tilde\beta$ such that $\tilde\beta(x)-\beta(x)\in M_{i_0}$ for all $x\in \mathrm{Im}\,\alpha$ (which will not affect the map $\psi$), we can assume that 
\[\mathrm{Im}\,\gamma=\mathrm{Im}\,\beta\circ\alpha.\]
Then by projecting (\ref{1-L-diagram}) to $\mathrm{Im}\,\gamma=\mathrm{Im}\,\beta\circ\alpha$,
with respect to the decomposition $U=\mathrm{Im}\,\gamma\oplus M$, we get
\[\gamma(F_{\mu(i)})=\beta\circ\alpha(F_{\mu(i)})\quad\mbox{for all $i=1,\ldots,m$,\ all $\mathcal{F}\in X$}.\]
Up to multiplying $\beta$ by a scalar, we can assume that the equality $\gamma=\beta\circ\alpha$ holds.
\end{proof}

\section{Construction of isomorphisms}

\label{section-4}

In this and the next section we prove Theorem \ref{T-main}.
Here we show that all pairs of ind-varieties that are claimed to be isomorphic in Theorem \ref{T-main} are indeed isomorphic.

We start with the following known fact.

\begin{lemma}
\label{L4.1}
{\rm (a)} Let $\mathcal{F}$ be a generalized flag in $V$ compatible with two bases $E$ and $E'$. Then the ind-varieties $\Flags(\mathcal{F},E,V)$ and $\Flags(\mathcal{F},E',V)$ are isomorphic.

{\rm (b)} Moreover, in the case where $V$ is endowed with an orthogonal (respectively, a symplectic) form $\omega$, $\mathcal{F}$ is isotropic, and $E$ and $E'$ are isotropic, then the ind-varieties $\OFlags(\mathcal{F},E,V)$ and $\OFlags(\mathcal{F},E',V)$ (respectively, $\SFlags(\mathcal{F},E,V)$ and $\SFlags(\mathcal{F},E',V)$) are isomorphic.
\end{lemma}

\begin{proof}
It suffices to construct a linear automorphism $\alpha:V\to V$ such that
$$
\alpha(E)=E',\qquad \forall F\in\mathcal{F},\, \alpha(F)=F,
$$
and $\alpha$ preserves the form $\omega$ in the situation (b) of the lemma.
Then $\alpha$ clearly induces an isomorphism $\mathcal{G}\mapsto \alpha(\mathcal{G})$ between the two considered ind-varieties.

(a) For $F\in\mathcal{F}$, we denote $E_F:=\{e\in E:e\in F\}$ and $\hat{E}_F:=E_F\setminus\bigcup_{\atop^{F'\in \mathcal{F}}_{F'\subset F}}E_{F'}$. We define similarly $E'_F$ and $\hat{E}'_F$.
Since the generalized flag $\mathcal{F}$ is $E$- and $E'$-compatible,
we have 
$$
F=\langle E_F\rangle=\langle\hat{E}_F\rangle\oplus \sum_{\atop^{F'\in \mathcal{F}}_{F'\subset F}}F'\quad\mbox{for all $F\in\mathcal{F}$.}
$$
This yields decompositions $E=\bigsqcup_{F\in\mathcal{F}}\hat{E}_F$, $E'=\bigsqcup_{F\in\mathcal{F}}\hat{E}'_F$
and, moreover, $|\hat{E}_F|=|\hat{E}'_F|=\dim F/(\sum_{\atop^{F'\in \mathcal{F}}_{F'\subset F}}F')$ for all $F\in\mathcal{F}$. Next, for every $F\in\mathcal{F}$, we can choose a bijection $\alpha_F:\hat{E}_F\to\hat{E}'_F$. This defines a bijection $\bigsqcup_{F\in\mathcal{F}}\alpha_F:E\to E'$, whose corresponding automorphism $\alpha:V\to V$ stabilizes each subspace of $\mathcal{F}$. 

(b) We adapt the construction made in (a) in the following way. 
The generalized flag $\mathcal{F}$ is equipped with the involution $F\mapsto F^\perp$,
and the bases $E$ and $E'$ are equipped with involutions $i_E:E\to E$ and $i_{E'}:E'\to E'$ satisfying the conditions of Definition \ref{D2.1}.
For every $F\in\mathcal{F}$ such that $F\subset F^\perp$ we have
$$\langle E_F\rangle=F\subset F^\perp=\langle E_{F^\perp}\rangle\subset V=F^\perp\oplus\langle i_E(E_F)\rangle$$
and for all $F\in\mathcal{F}$
 we have either $i_E(\hat{E}_F)\cap \hat{E}_F=\emptyset$ or $i_E(\hat{E}_F)=\hat{E}_F$; the latter equality holds for at most one $F$, namely the one, if it exists, such that $F^\perp \subsetneq F$ are consecutive subspaces in $\mathcal{F}$. The same applies to $E'$.
Then we have decompositions
$$
E=\bigsqcup_{\atop^{F\in\mathcal{F}}_{F\subset F^\perp}} i_E(\hat{E}_F)\cup \hat{E}_F,\qquad 
E'=\bigsqcup_{\atop^{F\in\mathcal{F}}_{F\subset F^\perp}} i_{E'}(\hat{E}'_F)\cup \hat{E}'_F.
$$
Now for all $F\in\mathcal{F}$ we can find  a bijection $\alpha_F:i_E(\hat{E}_F)\cup \hat{E}_F\to i_{E'}(\hat{E}'_F)\cup \hat{E}'_F$ such that $\alpha_F(i_E(e))=i_{E'}(\alpha_F(e))$ for all $e$. Whence we have a bijection $\alpha:E\to E'$ and, up to replacing the elements in $E'$ by suitable scalar multiples, we can assume that the corresponding automorphism $\alpha:V\to V$ preserves $\omega$. In addition, $\alpha(F)=F$ for all $F\in\mathcal{F}$, and this concludes the proof of the lemma.
\end{proof}


Next we show that $\Flags(\mathcal{F},E,V)$ and $\Flags(\mathcal{G},E',W)$ 
are isomorphic whenever
\begin{itemize}
\item[\rm (A)] $\mathcal{F}$ and $\mathcal{G}$ are isomorphic in the sense of Definition \ref{definition-isomorphic}\,(a), or
\item[\rm (B)] $\mathcal{F}$ is isomorphic to the dual generalized flag $\mathcal{G}^\perp$.
\end{itemize}

Assume first that $\mathcal{F}$ and $\mathcal{G}$ are isomorphic, hence there is an isomorphism $\phi:V\to W$ such that $\phi(\mathcal{F})=\mathcal{G}$. Then $E'':=\phi(E)$ is a basis of $W$, moreover $\mathcal{G}$ is compatible with $E''$, and the map $\phi$ induces an isomorphism of ind-varieties
$$
\Flags(\mathcal{F},E,V)\stackrel{\sim}{\to} \Flags(\mathcal{G},E'',W),\ \mathcal{F}'\mapsto \phi(\mathcal{F}').
$$
Thanks to Lemma \ref{L4.1}, we have an isomorphism $\Flags(\mathcal{G},E'',W)\cong\Flags(\mathcal{G},E',W)$. Altogether, we get an isomorphism $\Flags(\mathcal{F},E,V)\cong\Flags(\mathcal{G},E',W)$ as desired.

The case (B) is a consequence of (A), Lemma \ref{L4.1}, and the fact that the map $\mathcal{G'}\mapsto\mathcal{G}'^\perp$
clearly defines an isomorphism of ind-varieties $\Flags(\mathcal{G},E',W)\stackrel{\sim}{\to} \Flags(\mathcal{G}^\perp,E'^*,\langle E'^*\rangle)$ where $E'^*$ is the dual family of the basis $E'$.

Similar reasoning shows that $\OFlags(\mathcal{F},E,V)$ and $\OFlags(\mathcal{G},E',W)$
(resp.,  $\SFlags(\mathcal{F},E,V)$ and $\SFlags(\mathcal{G},E',W)$) are isomorphic ind-varieties whenever $\mathcal{F}$ and $\mathcal{G}$ are isomorphic in the sense of Definition \ref{definition-isomorphic}\,(b). Note that in this case we have $\mathcal{F}^\perp=\mathcal{F}$ and $\mathcal{G}^\perp=\mathcal{G}$.

\medskip

We now turn our attention to the additional isomorphisms from Theorem \ref{T-main}.
First, since in a symplectic space every line is isotropic, the isomorphism
between $X=\Grass(F,E,V)$, $Y=\SGrass(G,E',W)$, where $\dim F=\dim G=1$, is obvious.


\medskip

Finally, the isomorphism between $X=\OGrass(F,E,V)$, $Y=\OGrass(G,E',W)$, where $F^\perp=F$ and $\dim G^\perp/G=1$, is also easy to verify. The key observation is that the well-known isomorphism $\OGrass(n-1,\mathbb{C}^{2n-1})\cong\OGrass(n,\mathbb{C}^{2n})$
is compatible with standard extensions. More precisely,
thanks to Lemma \ref{L4.1}, we can assume without loss of generality that $W$ is a subspace of codimension one in $V$. Let
$$
E=\{e_1,\hat{e}_1,e_2,\hat{e}_2,\ldots,e_n,\hat{e}_n,\ldots\}
$$
be a basis of type D in $V$, with involution $i_E:e_i\mapsto \hat{e}_i$.
Let 
$F:=\langle e_i:i\geq 1\rangle$ so that $F=F^\perp$.
Consider also
$$
E'=\{e_1+\hat{e}_1,e_2,\hat{e}_2,\ldots,e_n,\hat{e}_n,\ldots\}
$$
which is a basis of type B in the subspace $W\subset V$, with involution $i_{E'}:e_i\mapsto \hat{e}_i$ for all $i\geq 2$ and $i_{E'}(e_1+\hat{e}_1)=e_1+\hat{e}_1$. Let $G:=\langle e_i:i\geq 2\rangle$, thus $\dim G^\perp/G=1$.

Set $V_n:=\langle e_1,\hat{e}_1,\ldots,e_n,\hat{e}_n\rangle$
and $W_n:=\langle e_1+\hat{e}_1,e_2,\hat{e}_2,\ldots,e_n,\hat{e}_n\rangle$.
We have exhaustions
$$
\cdots\hookrightarrow
\OGrass(n,V_n)\stackrel{\alpha_n}{\hookrightarrow} \OGrass(n+1,V_{n+1})\hookrightarrow\cdots,\quad X=\lim_{\to}\OGrass(m,V_m)
$$
and
$$
\cdots\hookrightarrow
\OGrass(n-1,W_n)\stackrel{\beta_n}{\hookrightarrow} \OGrass(n,W_{n+1})\hookrightarrow\cdots,\quad Y=\lim_{\to}\OGrass(m-1,W_m),
$$
where $\alpha_n:L\mapsto L\oplus\langle e_{n+1}\rangle$
and $\beta_n:M\mapsto M\oplus\langle e_{n+1}\rangle$.

For every $n$, there is an isomorphism
$$
\phi_n:\OGrass(n-1,W_n)\to \OGrass(n,V_n),\ M\mapsto\parbox[t]{6cm}{(the unique Lagrangian subspace $\hat{M}\in \OGrass(n,V_n)$ containing $M$).}
$$
Moreover, the  diagram
$$
\begin{array}{rcl}
\OGrass(n,V_n) & \stackrel{\alpha_n}{\hookrightarrow} & \OGrass(n+1,V_{n+1}) \\
\uparrow\mbox{\scriptsize $\phi_n$} & & \uparrow\mbox{\scriptsize $\phi_{n+1}$} \\
\OGrass(n-1,W_n) & \stackrel{\beta_n}{\hookrightarrow} & \OGrass(n,W_{n+1})
\end{array}
$$
is commutative. Indeed $\alpha_n\circ\phi_n(M)$ is a Lagrangian subspace in $\OGrass(n+1,V_{n+1})$ containing $M$ and $e_{n+1}$, thus containing $M\oplus\langle e_{n+1}\rangle=\beta_n(M)$, and therefore coinciding with $\phi_{n+1}\circ\beta_n(M)$.
Hence $X$ and $Y$ are isomorphic.

\section{Non-existence of other isomorphisms} \label{section-5}


In this section we complete the proof of Theorem \ref{T-main}. This is done by proving the following two statements.

\begin{theorem}
\label{T5.1}
Assume that $X$, $Y$ is a pair of ind-varieties of generalized flags of the same type (general, orthogonal, or symplectic), different from the pair 
$$
\OGrass(F,E,V),\ \OGrass(G,E',W)\quad\mbox{with $\dim F^\perp/F=0$, $\dim G^\perp/G=1$, or vice versa.}
$$
Consider two arbitrary (possibly isotropic) generalized flags $\mathcal{F}\in X$ and $\mathcal{G}\in Y$.
Then $X$ is isomorphic to $Y$ if and only if $\mathcal{F}$ is isomorphic to $\mathcal{G}$ or to $\mathcal{G}^\perp$.
\end{theorem}

\begin{theorem}
\label{T5.2}
Assume $X$, $Y$ are two ind-varieties of generalized flags of different types.
Then $X$ is isomorphic to $Y$ if only if $X$, $Y$ (or $Y$, $X$) is the pair 
\begin{equation}
\label{newequation2}
\Grass(F,E,V),\ \SGrass(G,E',W)\quad\mbox{with $\dim F=\dim G=1$ or $\dim V/F=\dim G=1$}.
\end{equation}
\end{theorem}

The direct implications in Theorems \ref{T5.1}--\ref{T5.2} are shown in Section \ref{section-4}. It remains to prove the reverse implications.

We start with some auxiliary results. 
By $\mathbb{P}^n$ we denote the $n$-dimensional projective space and by $\mathbb{P}^\infty$ we denote the infinite-dimensional projective ind-space: $\mathbb{P}^\infty=\lim\limits_\to\mathbb{P}^n$.
The following proposition extends Theorem 2 from \cite{PT1}.

\begin{proposition}
\label{P5.3}
Let $X$, $Y$ be two ind-grassmannians, so $X=\Grass(F,E,V)$, $X=\OGrass(F,E,V)$, or $X=\SGrass(F,E,V)$, and $Y=\Grass(G,E',W)$, $Y=\OGrass(G,E',W)$, or $Y=\SGrass(G,E',W)$.
If $Y=\OGrass(G,E',W)$, we assume that  $\dim G^\perp/G\notin\{0,1\}$.
Then $X$ is isomorphic to $Y$ if and only if one of the following condition holds.
\begin{itemize}
\item[(A)] $X=\Grass(F,E,V)$ and $Y=\SGrass(G,E',W)$ with $\dim F=\dim G=1$ or $\dim V/F=\dim G=1$ (or vice versa).
\item[(B)] $X$ and $Y$ are of the same type with $\dim F=\dim G$, or $X=\Grass(F,E,V)$ and $Y=\Grass(G,E',W)$ with $\dim V/F=\dim G<\infty$ (or vice versa).
\end{itemize}
\end{proposition}

\begin{proof}
The case where $X$ is different from $\OGrass(F,E,V)$ with $\dim F^\perp/F\in\{0,1\}$ is treated in \cite[Theorem 2]{PT1}. Hence it remains to show that $X\not\cong Y$ whenever $X=\OGrass(F,E,V)$ with $\dim F^\perp/F\in\{0,1\}$.

We will do this by the same method used in \cite{PT1}.
Indeed, using results in \cite[Section 4]{Landsberg-Manivel}, it is not difficult to check that through any point $x\in X$, there is a family $\mathcal{P}^3$ consisting of maximal $3$-dimensional linearly embedded projective subspaces of $X$, and a family $\mathcal{P}^\infty$ of maximal linearly embedded infinite-dimensional projective ind-spaces. 
Moreover, the intersection of any space in $\mathcal{P}^3$ with a space in $\mathcal{P}^\infty$ is isomorphic to $\mathbb{P}^2$ or equals the point $x$.

We claim that this type of configuration of maximal linearly embedded projective spaces passing through a point does not appear on any ind-grassmannian $Y$. Indeed, it is well known that $Y$ admits a linear embedding into an ind-grassmannian of general type (this embedding being the identity of $Y$ itself is of general type). Using an appropriate such embedding, it is easy to check that
 the complete list of ind-grassmannians $Y$ having a family of maximal linearly embedded projective spaces $\mathbb{P}^3$ and a family of maximal linearly embedded projective spaces $\mathbb{P}^\infty$ passing through a fixed point $y\in Y$ is \begin{itemize}
\item $\Grass(F,E,V)$ where $\dim F=3$ or $\dim V/F=3$,
\item $\OGrass(F,E,V)$ where $\dim F=3$,
\item $\OGrass(F,E,V)$ where $\dim F^\perp/F\in\{6,7\}$,
\item $\SGrass(F,E,V)$ where $\dim F=3$,
\item $\SGrass(F,E,V)$ where $\dim F^\perp/F=2$.
\end{itemize}
However, in all these cases, the intersection of a maximal linearly embedded space $\mathbb{P}^3$ and a maximal linearly embedded ind-space $\mathbb{P}^\infty$ passing through the same point $y$ is isomorphic to $\mathbb{P}^1$ or is the point $y$ itself.
This proves our claim.
\end{proof}

\begin{lemma} \label{lemma-5.4new}
Let $X=\Flags(\mathcal{F},E,V)$, $X=\OFlags(\mathcal{F},E,V)$, or $X=\SFlags(\mathcal{F},E,V)$,
and let $\pi:X\to Y$ be an $\mathrm{Aut}\,X$-equivariant smooth surjective morphism, where $Y$ is another ind-variety of generalized flags. Then $Y$ is isomorphic respectively to $\Flags(\mathcal{F}',E,V)$, $\OFlags(\mathcal{F}',E,V)$, or $\SFlags(\mathcal{F}',E,V)$, where $\mathcal{F}'$ is a generalized subflag of $\mathcal{F}$.
\end{lemma}

\begin{proof}
Consider an exhaustion of $X$ by standard extensions
$$
\cdots \hookrightarrow X_n\hookrightarrow X_{n+1}\hookrightarrow \cdots,\ X=\lim_\to X_m
$$
and a corresponding exhaustion of $Y$
\begin{equation}
\label{Y}
\cdots \hookrightarrow Y_n:=\pi(X_n)\hookrightarrow Y_{n+1}:=\pi(X_{n+1})\hookrightarrow \cdots,\quad Y=\lim_\to Y_m.
\end{equation}
Since $X$ is a homogeneous space for the finitary ind-group $\lim\limits_\to\mathrm{Aut}^0X_n$, where $\mathrm{Aut}^0$ stands for connected component of unity (see also \cite{IP} for a description of $\mathrm{Aut}\,X$), any automorphism in $\mathrm{Aut}^0X_n$ extends to an automorphism of $X$. Therefore each projection $\pi_n:X_n\to Y_n$ is $\mathrm{Aut}^0X_n$-equivariant.
Through the theory of finite-dimensional flag varieties, this implies that $Y_n$ is isomorphic to a shorter flag variety $X'_n$ of same type as $X_n$ so that $\pi_n$ corresponds to the natural projection $X_n\to X'_n$.
The standard extensions $X_n\hookrightarrow X_{n+1}$ then induce an exhaustion by standard extensions
$$
\cdots \hookrightarrow X'_n\hookrightarrow X'_{n+1}\hookrightarrow \cdots,\quad X':=\lim_\to X'_m
$$
which commutes with the exhaustion (\ref{Y}) via the isomorphisms $Y_n\cong X'_n$. Therefore, $Y$ is isomorphic to the ind-variety of generalized flags $\lim\limits_\to X'_n$ of the form indicated in the statement.
\end{proof}

We can now prove Theorems \ref{T5.1} and \ref{T5.2}.

\begin{proof}[Proof of Theorem \ref{T5.2}]
First we suppose that $X=\Flags(\mathcal{F},E,V)$ and $Y=\SFlags(\mathcal{G},E',W)$ are isomorphic ind-varieties. By Lemma \ref{lemma-5.4new}, for every nonzero proper subspace $F\in \mathcal{F}$ there is an isotropic subspace $G\in\mathcal{G}$ such that $\Grass(F,E,V)\cong\SGrass(G,E',W)$, and for every nontrivial isotropic subspace $G'\in\mathcal{G}$ there is a subspace $F'\in\mathcal{F}$ such that $\SGrass(G',E',W)\cong\Grass(F',E,V)$. Since an isomorphism $\Grass(F,E,V)\cong\SGrass(G,E',W)$
can exist only if $\dim G=1$ and $\dim F=1$ or $\dim V/F=1$ (see Proposition \ref{P5.3}),
we infer that $\mathcal{G}$ must be of the form $\mathcal{G}=\{\{0\}\subset G\subset G^\perp\subset W\}$ with $\dim G=1$, while $\mathcal{F}$ could be only  of the form $\mathcal{F}=\{\{0\}\subset F\subset V\}$ with $\dim F=1$ or $\dim V/F=1$, or of the form $\mathcal{F}=\{\{0\}\subset F_1\subset F_2\subset V\}$ with $\dim F_1=\dim V/F_2=1$. However, the latter is not an option since we would then have $\Pic\,X=\mathbb{Z}^2\not\cong\mathbb{Z}=\Pic\,Y$. The conclusion of the theorem follows.

Next we suppose that $X=\Flags(\mathcal{F},E,V)$, or $X=\SFlags(\mathcal{F},E,V)$, and $Y=\OFlags(\mathcal{G},E',W)$ are isomorphic. Arguing as in the first case, it  suffices to note that an isotropic ind-grassmannian $\OGrass(G,E',W)$ is never isomorphic to an ind-grassmannian $\Grass(F,E,V)$ or $\SGrass(F,E,V)$. Thus the claim follows again from Proposition \ref{P5.3}.
\end{proof}




\begin{proof}[Proof of Theorem \ref{T5.1}]
Assume that two ind-varieties $X$ and $Y$ as in the statement of Theorem \ref{T5.1} are isomorphic. Without loss of generality we may suppose
that $X=\Flags(\mathcal{F},E,V)$, or  $X=\OFlags(\mathcal{F},E,V)$, or $X=\SFlags(\mathcal{F},E,V)$,
and $Y=\Flags(\mathcal{G},E',V)$, or $Y=\OFlags(\mathcal{G},E',V)$, or $Y=\SFlags(\mathcal{G},E',V)$, respectively.
Fix an isomorphism $\varphi:X\to Y$ with $\varphi(\mathcal{F})=\mathcal{G}$. In the case of ind-varieties of orthogonal generalized flags, we first assume that $\mathcal{F}$ and $\mathcal{G}$ do not contain isotropic subspaces $F$ with $\dim F^\perp/F\leq 2$. According to our convention from Section \ref{section-2}, this is equivalent to assuming that both $\mathcal{F}$ and $\mathcal{G}$ do not contain maximal isotropic subspaces.

The existence of the isomorphism $\varphi$ implies the existence of a commutative diagram
\begin{equation}
\label{diagram1}
\xymatrix{X_1 \ar@{^{(}->}[r]^{\psi_1} \ar@{^{(}->}[d]^{\varphi_1} & X_2 \ar@{^{(}->}[r]^{\psi_2} \ar@{^{(}->}[d]^{\varphi_2} & X_3 \ar@{^{(}->}[r]^{\psi_3} \ar@{^{(}->}[d]^{\varphi_3} & \cdots \\
Y_1  \ar@{^{(}->}[r]^{\psi'_1} \ar@{^{(}->}[ru]^{\xi_1} & Y_2 \ar@{^{(}->}[r]^{\psi'_2} \ar@{^{(}->}[ru]^{\xi_2} & Y_3 \ar@{^{(}->}[r]^{\psi'_3} \ar@{^{(}->}[ru]^{\xi_3} & \cdots }
\end{equation}
where all the maps are embeddings and the rows are exhaustions of $X$ and $Y$, respectively, by standard extensions.

We claim that $\varphi_i$ and $\xi_i$ are standard extensions for all $i\geq 1$. First we check that $\varphi_i$ and $\xi_i$ are linear. To do this, we analyse the maps induced on  Picard groups.
Since $\xi_i^*\circ\varphi_{i+1}^*=\psi'^*_i$ and $\psi'^*_i$ is surjective, it follows that $\xi_i^*$ is surjective. Letting $[M]$ be one of the preferred generators of $\Pic\,Y_i$, there is $a\in\Pic\,X_{i+1}$ with $\xi_i^*a=[M]$. Due to Lemma \ref{L-3.2}  we can choose $a=[L]$, where $[L]$ is one of the preferred generators of $\Pic\,X_{i+1}$. Then $\varphi_i^*[M]=\psi_i^*[L]$ should be equal to $0$ or be a preferred generator of $\Pic\,X_i$ because $\psi_i$ is linear. This shows that $\varphi_i$ is linear, and a similar argument works for $\xi_i$.
Next, since $\psi_i=\xi_i\circ\varphi_i$ and $\psi_i$ is a standard extension, we deduce from Theorem \ref{T-PT}  that $\varphi_i$ and $\xi_i$ do not factor through a direct product or a maximal quadric (in the orthogonal case). Now Theorem \ref{T-PT} implies that $\varphi_i$ and $\xi_i$ are standard extensions as long as all maps induced on grassmannians by $\varphi_i$ and $\xi_i$, see Proposition \ref{prop-linear}, are standard extensions. However, if some of these maps induced on grassmannians were not a standard extension, then by Theorem \ref{T-PT}\,(b) it would be necessarily a combination of isotropic and standard extensions. Now the commutativity of diagram (\ref{diagram1}), together with \cite[Lemma 3.9]{PT1}, implies that some map induced on grassmannians by $\varphi_i$ or $\xi_i$ would be a combination of isotropic and standard extensions. By \cite[Remark 3.10]{PT1}, this contradicts the assumption that $\psi_i$ are standard extensions. This establishes the claim that $\varphi_i$ and $\xi_i$ are standard extensions for all $i\geq 1$.

In the isotropic case, $\varphi_i$ and $\xi_i$ are strict standard extensions. Let us show that also for $X$ and $Y$ of general type we can reduce the problem to the case where $\varphi_i$ and $\xi_i$ are strict standard extensions. 
Indeed, note that there is also the diagram
\begin{equation}
\label{diagram2}
\xymatrix{
Y_1  \ar@{^{(}->}[r]^{\psi'_1} 
\ar[d]_\sim^{\delta_1} & Y_2 \ar@{^{(}->}[r]^{\psi'_2} 
\ar[d]_\sim^{\delta_2} & Y_3 \ar@{^{(}->}[r]^{\psi'_3} 
\ar[d]_\sim^{\delta_3} & \cdots \\
{Y_1}^\vee \ar[r]^{{\psi'_1}^\vee} & {Y_2}^\vee \ar@{^{(}->}[r]^{{\psi'_2}^\vee} & {Y_3}^\vee \ar@{^{(}->}[r]^{{\psi'_3}^\vee} & \cdots}
\end{equation}
where ${Y_i}^\vee$ and $\delta_i$ are as Definition \ref{D-linear}\,(b.2); the bottom line of the diagram forms an exhaustion of $\mathrm{Fl}(\mathcal{G}^\perp,E'^*,\langle E'^*\rangle)$ (see the notation in Section \ref{section-2.2}).
Up to composing the two diagrams (\ref{diagram1}) and (\ref{diagram2}), thus considering $\mathrm{Fl}(\mathcal{G}^\perp,E'^*,\langle E'^*\rangle)=\lim\limits_\to{Y_n}^\vee$ instead of $\mathrm{Fl}(\mathcal{G},E',V)$, we can assume that $\varphi_1$ is a strict standard extension. Then it follows from Lemma \ref{L-se}\,(b) that every map $\varphi_i$, $\xi_i$ is a strict standard extension.

Let now $V=\bigcup_{n\geq 1}V_n$ be an exhaustion of the vector space $V$ so that $X_n$ and $Y_n$ are varieties of flags in $V_n$.
By Definition \ref{D-linear}\,(b.1), the strict standard extensions in (\ref{diagram1}) are induced by a diagram of injective linear maps
\begin{equation}
\label{diagram3}
\xymatrix{V_1 \ar[r]^{\iota_1} \ar[d]^{\alpha_1} & V_2 \ar[r]^{\iota_2} \ar[d]^{\alpha_2} & V_3 \ar[r]^{\iota_3} \ar[d]^{\alpha_3} & \cdots \\
V_1  \ar[r]^{\iota'_1} \ar[ru]^{\beta_1} & V_2 \ar[r]^{\iota'_2} \ar[ru]^{\beta_2} & V_3 \ar[r]^{\iota'_3} \ar[ru]^{\beta_3} & \cdots }
\end{equation}
where the $\iota_n:V_n\to V_{n+1}$ are the inclusion maps, and the maps $\alpha_n$ are induced by the morphisms $\varphi_n$ and make the diagram commutative if one disregards the diagonal arrows. Moreover, by Lemma \ref{L-diagram},
the maps $\beta_n$ can be chosen so that the entire diagram (\ref{diagram3}) is commutative. The maps $\iota'_n$ are then simply the compositions $\iota'_n=\alpha_{n+1}\circ\beta_n$.
In the orthogonal and symplectic cases the maps $\beta_n$ are unique and $\iota'_n=\iota_n$ (see Remark \ref{remark-3.4}).

Therefore, this diagram induces a linear isomorphism $\alpha:V\to V'$ where $V'$ is the direct limit of the lower row of the diagram, such that the generalized flags $\alpha(\mathcal{F})$ in $V'$ and $\mathcal{G}$ in $V$ coincide as points of $Y$.
In the orthogonal and symplectic cases, we have simply $\alpha(\mathcal{F})=\mathcal{G}$, and hence $\mathcal{F}$ and $\mathcal{G}$ are isomorphic.

If $X$ and $Y$ are of general type, we need an additional argument.
Note that in this case we have $Y=\Flags(\alpha(\mathcal{F}),\alpha(E),V')$.
Since the generalized flags $\mathcal{F}$ and $\alpha(\mathcal{F})$ are isomorphic via the linear operator $\alpha$, it remains to show that the generalized flags $\alpha(\mathcal{F})$ and $\mathcal{G}$ are isomorphic. This follows from the easy observation that if a generalized flag $\mathcal{H}$ in $V$ and a generalized flag $\mathcal{H}'$ in $V'$ represent the same point in  $Y$, then $\mathcal{H}$ and $\mathcal{H}'$ are isomorphic.
Indeed, one builds inductively an isomorphism of direct systems 
\begin{equation*}
\xymatrix{V_1 \ar[r]^{\iota_1} \ar[d]^{\eta_1} & V_2 \ar[r]^{\iota_2} \ar[d]^{\eta_2} & V_3 \ar[r]^{\iota_3} \ar[d]^{\eta_3} & \cdots \\
V_1  \ar[r]^{\iota'_1}  & V_2 \ar[r]^{\iota'_2}  & V_3 \ar[r]^{\iota'_3}  & \cdots }
\end{equation*}
such that $\eta_i(\mathcal{H}\cap V_i)=\mathcal{H}'\cap V_i$. Then the direct limit map $\lim\limits_\to\eta_n$ provides an isomorphism between $\mathcal{H}$ and $\mathcal{H}'$.
This allows us to conclude that $\mathcal{F}$ and $\mathcal{G}$ are isomorphic generalized flags.

To finish the proof of the theorem, it remains to consider the orthogonal case where
$$
X=\OFlags(\mathcal{F},E,V)\quad\mbox{and}\quad Y=\OFlags(\mathcal{G},E',V),
$$
and where one of the generalized flags $\mathcal{F}$ or $\mathcal{G}$ contains a maximal isotropic subspace.
There is no loss of generality in assuming that $X\not=\OGrass(F,E,V)$ and $Y\not=\OGrass(G,E',V)$.

Consider first the case where $\mathcal{F}=\{\{0\}\subset F\subset\tilde{F}\subset F^\perp\subset V\}$ with $\dim \tilde{F}/F=1$ and $\tilde{F}^\perp=\tilde{F}$. Then $\Pic\,X\cong\mathbb{Z}^2$ and hence $\mathcal{G}=\{\{0\}\subset G\subset\tilde{G}\subseteq\tilde{G}^\perp\subset G^\perp\subset V\}$ with $\dim\tilde{G}^\perp/\tilde{G}\in\{0,1\}$. If $\mathcal{G}$ is not isomorphic to $\mathcal{F}$, then $\dim \tilde{G}/G\geq 2$ or $\dim \tilde{G}^\perp/\tilde{G}=1$. In both cases $Y$ admits a proper smooth surjection to $\OGrass(G,E',V)$, while the only orthogonal ind-grassmannian to which $X$ admits a proper smooth surjection is $\OGrass(\tilde{F},E,V)$ with $\tilde{F}^\perp=\tilde{F}$.
Since $\OGrass(G,E',V)$ is not isomorphic to $\OGrass(\tilde{F},E,V)$ by Proposition \ref{P5.3}, this case is settled.
 
Now we consider the case of arbitrary orthogonal generalized flags $\mathcal{F}$ and $\mathcal{G}$ containing respective maximal isotropic subspaces.
Define projections as follows:
\begin{itemize}
\item $\pi_X:X\to\hat{X}$ where
$\hat{X}:=\OFlags(\hat{\mathcal{F}},E,V)$ is the ind-variety of
generalized flags associated to
$\hat{\mathcal{F}}:=\mathcal{F}\setminus\{F,F^\perp:F\in\mathcal{F},\ \dim F^\perp/F\leq 2\}$,
\item $\pi_Y:Y\to\hat{Y}$ where
$\hat{Y}:=\OFlags(\hat{\mathcal{G}},E',V)$ is the ind-variety of
generalized flags associated to
$\hat{\mathcal{G}}:=\mathcal{G}\setminus\{G,G^\perp:G\in\mathcal{G},\ \dim G^\perp/G\leq 2\}$,
\end{itemize}
We can assume without loss of generality that $\hat{X}$ and $\hat{Y}$ are both proper ind-varieties of generalized flags (not points) because otherwise we land in the case already considered.
(The case of 
$X=\OFlags(\mathcal{F},E,V)$ for $\mathcal{F}=\{\{0\}\subset F\subset\tilde{F}\subset F^\perp\subset V\}$ where $\dim \tilde{F}/F=1$ and $\tilde{F}^\perp=\tilde{F}$,
and $Y=\OGrass(\mathcal{G},E',V)$ for $\mathcal{G}=\{\{0\}\subset G\subset G^\perp\subset V\}$ with $\dim \tilde{G}^\perp/\tilde{G}= 1$ is ruled out by the existence of the isomorphism $\varphi$ because $\Pic\,X\cong\mathbb{Z}^2$ and $\Pic\,Y\cong\mathbb{Z}$.)

By Lemma \ref{lemma-5.4new} the isomorphism $\varphi$ induces an isomorphism $\hat\varphi:\hat{X}\to \hat{Y}$
with $\hat\varphi(\hat{\mathcal{F}})=\hat{\mathcal{G}}$.
Now the first part of the proof allows us to conclude that there exists an automorphism 
$$
\mbox{$\hat\varphi_V:V\to V$, preserving $\omega$ and such that $\hat{\mathcal{G}}=\hat\varphi_V(\hat{\mathcal{F}})$}.
$$
We claim that $\hat\varphi_V(\mathcal{F})=\mathcal{G}$, implying that the isotropic generalized flags $\mathcal{F}$ and $\mathcal{G}$ are isomorphic. Indeed, the maximal isotropic space $\tilde{F}\in\mathcal{F}$ is the union of all subspaces $F''\subsetneq\tilde{F}$ with the property that $F''$ belongs to some point $\mathcal{F}''\in X$ and has codimension $2$ or more in $\tilde{F}$. A similar statement applies to the maximal isotropic space $\tilde{G}\in\mathcal{G}=\varphi(\mathcal{F})$. Therefore,  $\tilde{G}$ equals the union of the spaces $\hat\varphi_V(F'')$ and hence coincides with $\hat\varphi_V(\tilde{F})$.
The same argument applies to spaces $F\in\mathcal{F}$ of codimension 1 in $\tilde{F}$ and $G\in\mathcal{G}$ of codimension 1 in $\tilde{G}$, if they exist, i.e., $\hat\varphi_V(F)=G$. Therefore, $\hat\varphi_V(\mathcal{F})=\mathcal{G}$.
\end{proof}

\begin{remark}
Note that the appearance of the space $V'$ in the proof of Theorem \ref{T5.1} correlates well with the results in \cite{IP} on automorphism groups of ind-varieties of generalized flags. It is essential that $V'$ does not have to coincide with $V$. Indeed, let $X=\Grass(F,E,V)=Y$ for $\dim V/F=1$. Then $\mathrm{Aut}\,X=\mathrm{PGL}(\langle E^*\rangle)$ by \cite{IP}. Consider the automorphism $\varphi:X\to X$ defined with respect to the basis $E^*$ by the infinite $\mathbb{Z}_{>0}\times\mathbb{Z}_{>0}$ matrix with entries equal to $1$ in the entire first row and the diagonal, and equal to $0$ elsewhere. Then one can compute that  $V'=\langle e_1-\sum_{i\geq 2}e_i,e_2,e_3,\ldots \rangle$
where the vectors in $V$ are interpreted as linear functions on $\langle E^*\rangle$. 
\end{remark}




\end{document}